\documentclass[11pt]{article}

\usepackage{pdfsync}

\usepackage{a4wide}
\usepackage{amsmath}
\usepackage{amsfonts}

\usepackage{hyperref}

\newcommand{\R}{\mathbb{R}}
\newcommand{\ep}{\varepsilon}
\newcommand{\dt}{\partial_t}
\newcommand{\e}{\mathbf{e}_{3}}
\newcommand{\one}{\mathbf{1}}

\newtheorem{lemma}{Lemma}[section]

\newtheorem{theorem}{Theorem}[section]

\newtheorem{definition}{Definition}[section]
\newtheorem{examp}{Example}[section]

\newenvironment{proof}
 {\begin{trivlist}\item[\textit{ Proof:}]\,}
 {$\Box$\end{trivlist}}
\newcommand{\produ}[1]{\langle #1\rangle}
\newcommand{\RR}[1]{\mathrm{e}^{J#1}}

\title{Associated Families of Surfaces in Warped Products and Homogeneous Spaces}
\author{
Marie-Am\'elie Lawn \\
\small Department of Mathematics, Imperial College, London (UK) \\
\small m.lawn@imperial.ac.uk \\[1em]
Miguel Ortega\\
\small Institute of Mathematics IEMath-GR, Department of Geometry and Topology,\\
\small University of Granada, Granada, 18071 SPAIN\\
\small miortega@ugr.es
}
\date{\today}

\begin{document}
\maketitle
\begin{abstract}
We classify Riemannian surfaces admitting associated families in three dimensional homogeneous spaces with four-dimensional isometry groups and in a wide family of (semi-Riemannian) warped products, with an extra natural condition (namely, rotating structure vector field).  We prove that, provided the surface is not totally umbilical, such families exist in both cases if, and only if, the ambient manifold is a product and the surface is minimal. In particular, there exists no associated families of surfaces with rotating structure vector field in the Heisenberg group.
\end{abstract}


\section{Introduction}
Classically, associated families are certain isometric deformations of minimal surfaces in Euclidean three-space, the best known example being the deformation of the catenoid into the
helicoid. A well-known property for such isometric immersions $\chi:M\rightarrow \mathbb{R}^3$ is namely the existence of a so called strong  associated family, i.e. a one-parameter family $\chi_{\theta}$, with $\theta\in S^1$ and $\chi_0=\chi$, of isometric immersions "rotating the differential" and preserving, at every point, the tangent plane and the Gau{\ss} map. More precisely, denoting by $R_{\theta}:TM\rightarrow TM$ the rotation of the tangent plane by an angle $\theta$, a strong  associated family of isometric immersions is a smooth family $\chi_{\theta}$ satisfying $d\chi(R_{\theta})=d\chi_{\theta}$. Moreover it is known that minimality and the existence of an  associated family are equivalent conditions.

A generalization to a larger class of surfaces is given by constant mean curvature surfaces (shortly CMC) and can also be characterized by the existence of an  associated family $\chi_{\theta}$ which is defined as follows. Let $\chi_0: M \to \mathbb{R}^3$ be a CMC surface and let 
$\chi_{\theta}:M\rightarrow \mathbb{R}^3$ be a smooth family of isometric immersions into $\mathbb{R}^3$  with second fundamental forms $A_{\theta}$. Then $\chi_{\theta}$ is an  associated family if, and only if, $A_{\theta}=R_{-\theta}A_0R_{\theta}$. If $M$ is minimal, the operator $R_{\theta}$ anticommutes with $A_0$ and we recover the above strong  associated family. It is a well-known fact that there exists an  associated family if and only $M$ is a CMC surface, which is furthermore equivalent to the harmonicity of the Gau{\ss} map. More generally,  an analogous result holds for CMC surfaces in three-dimensional space forms. It is worth pointing out that the notion of (strong)  associated family was extended in \cite{BEFT} to K\"ahler manifolds in $\mathbb{R}^n$ without significant modifications (see also \cite{E}). The existence of an  associated family for such manifolds is then equivalent to the pluriharmonicity of the Gau{\ss} map.

In recent years, minimal and CMC surfaces in other ambient spaces, such as homogeneous three-spaces and (Lorentzian) warped products, have received a lot of attention. Especially, minimal surfaces in the product spaces $\mathbb{S}^2\times\mathbb{R}$ and $\mathbb{H}^2\times\mathbb{R}$ have been extensively studied and the existence of an  associated minimal family was showed in \cite{D2} (see also \cite{H} and \cite{HST}, and \cite{R} for the Lorentzian case). Nevertheless, no equivalence between the minimality of the surface and the existence of the family had been proven until now. We also mention \cite{K} and \cite{LR} where  associated families of minimal surfaces in (multi)products of space forms where studied.

Daniel considered in \cite{D1} surfaces immersed into the 3-dimensional homogeneous spaces  with four-dimensional isometry group $\mathbb{E}(\kappa,\tau)$. These spaces are well-known to be Riemannian fibrations over a 2-dimensional space form, where $\kappa$  is the curvature of the base surface of the fibration, and $\tau$ the bundle curvature. Daniel showed the existence of a Lawson-type correspondence of CMC surfaces, the so-called \textit{sister surfaces}, but the question of the existence of the  associated family for minimal or CMC surfaces in such geometries when $\tau\neq 0$ remained open. Of particular interest is the case of the Heisenberg group (see for instance \cite{D1}, Remark 5.10). 

Our aim is the classification of surfaces admitting an  associated family in the ambient space $P^3$, where $P^3$ is either a Thurston geometry with four-dimensional isometry group or a semi-Riemannian warped product of the form $P^3=\varepsilon\,I\times\mathbb{M}_k^2$, where $\ep=\pm 1$, $a:I\subset\R\rightarrow\R^+$ is the scale factor and $\mathbb{M}_k^2(c)$, $k\in\{0,1\}$, is the semi-Riemannian space form of index $k$ and constant curvature $c$, excluding in our study the well-understood case where the ambient space $P^3$ is a space form. In fact, these two cases are similar because of the importance in the Thurston Geometries of the unit vertical Killing vector field $\dt$ tangent to the fibers, which corresponds in the case of warped products to the vector in the direction of the factor $\mathbb{R}$. Given a hypersurface $M$, the vector field $\dt$ can be decomposed along $M$ in its tangent and normal parts i.e. $\dt=T+f\nu$, where $\nu$ is the unit vector field normal to $M$. The compatibility equations depend in both cases not only on the shape operator $A$, but also on the tangent vector field $T$ and on the normal component function $f$. Moreover, whereas the Gau{\ss} and Codazzi integrability conditions are well-known to be necessary and sufficient for the existence of isometric immersions of hypersurfaces into space forms,  it was proved in \cite{D1} for the homogeneous case and in \cite{LO} for the warped product case that two additional equations involving the covariant derivative of $T$ and $f$ have to be satisfied on $M$ in order for the immersion to exist.

 Consequently, considering  associated families of surfaces in those spaces, it is necessary to not only rotate the shape operator, but also the vector field $T$. We introduce for that purpose a more general transformation and we call a smooth family $\chi_{\theta}:M\rightarrow P^3$ with second fundamental form $A_{\theta}$ and vector $T_{\theta}$ a \textit{generalized  associated family with rotating structure field $T_{\theta}$,} if, and only if, there exist smooth real functions $F_1$, $F_2$, $\lambda$ and $\mu$, such that
\[
A_{\theta}=F_1(\theta)R_{-\theta}A^aR_{\theta}+F_2(\theta)A^c,\quad 
T_{\theta}=\lambda(\theta) T+\mu(\theta) JT,\]
where $J$ is the rotation of angle $\frac{\pi}{2}$ on $M$ given by the orientation, $R_{\theta}$ is the operator given by $R_{\theta}X = \cos(\theta)X+\sin(\theta)JX$, $X\in TM$, and $A^a$ (resp. $A^c$) are the parts of $A$ anticommuting (resp. commuting) with $J$.  The shape operator $A$ can be uniquely decomposed as $A=A^c+A^a$ where $A^cJ=JA^c$ and $A^aJ=-JA^a$. Note that $R_{-\theta}AR_{\theta} = A^c+R_{-\theta}A^aR_{\theta}$ for any $\theta$. This means that, until now, the crutial part is the anti-commutative one, or rather, the commutative part provides little information. Thus, if  we introduce some auxiliary functions depending only in the parameter $\theta$ and not on the point of the surface, we can also make the commutative part more important, obtaining in addition the same results.

We are then able to prove that, in the case where $P^3$ is an homogeneous space, there exists a generalized  associated family if, and only if, $P^3$ is a product and $M$ is minimal, and in the case of warped products that the existence of such deformations is also equivalent to either $P^3$ being a product and $M$ minimal, or $M$ being totally umbilical. In all cases, the generalized  associated family turns out to be the classical  associated family (with vector $T_{\theta}=e^{-2J\theta}T$). In particular we show that there exists no  associated family of minimal surfaces in the Heisenberg group.

\section{Basics}

Let $(M,g)$ be a connected, oriented Riemannian surface. It is well-known that it admits a complex structure $J$, so that it becomes a Riemann surface $(M,g,J)$. We consider an isometric immersion $\chi:(M,g)\rightarrow (N^3,\produ{,})$ in a semi-Riemannian 3-dimensional manifold, with  second fundamental form $\sigma$. Let $\overline{\nabla}$ be the Levi-Civita connection of $N^3$. Since $M$ is an oriented hypersurface in $N^3$, given a unit normal vector field $\e$, with sign $\ep_3=\produ{\e,\e}=\pm 1$, we have the corresponding shape operator $A$.  For each $\theta\in\mathbb{R}$, we define the operator on $M$ by $R_{\theta}X=\cos(\theta)X+\sin(\theta)JX=e^{J\theta}X$,  for any $X\in TM$. We consider now smooth one-parameter families $\chi_{\theta}:(M,g)\longrightarrow (N,\produ{,})$ of isometric immersions, with their corresponding unit normal vector fields $\e^{\theta}$, signs $\ep^{\theta}_3=\produ{\e^{\theta},\e^{\theta}}=\pm 1$ and shape operators $A_{\theta}$. For this family to be smooth, $\ep_3=\ep_3^{\theta}$ for any $\theta \in \R$.
\begin{definition} Such a family will be called \textrm{ associated family with $\chi$} if their shape operators satisfy $A_{\theta} = 
\mathrm{e}^{-J\theta}A\RR{\theta}$.
\end{definition}
The classical definition  states that the second fundamental forms $\sigma_{\chi_{\theta}}$ satisfy 
$\sigma_{\chi_{\theta}}(X,Y)=\sigma(R_{\theta} X, R_{\theta} Y),$ which is clearly equivalent to the above definition. For further references see for example \cite{E}.

\begin{definition}\label{generalized} A smooth map $\tilde{\chi}:\R\times M\rightarrow N$ will be called {\normalfont a generalized  associated family with $\chi$} if the following conditions hold:
\begin{enumerate}
\item $\tilde{\chi}(0,\cdot)=\chi$;
\item for each $\theta\in\R$, $\chi_{\theta}=\tilde{\chi}(\theta,\cdot):(M,g)\rightarrow(N,\produ{,})$ is a smooth isometric immersion;
\item there exist two smooth functions $F_1,F_2:\R\rightarrow\R$ satisfying $F_1(0)=F_2(0)=1$ and the shape operators satisfy $A_{\theta} = F_1(\theta)\mathrm{e}^{-J\theta}A^a\RR{\theta}+F_2(\theta)A^c$ for any $\theta\in\R$.
\end{enumerate}
\end{definition}
Note that whenever $F_1(\theta)=F_2(\theta)	=1$ everywhere,  we recover the action $A_{\theta}=\mathrm{e}^{-J\theta}A\RR{\theta}$.

Let $\nabla$ be the Levi-Civita connection of $M$. We call $\sigma_{\theta}$ the second fundamental form of $\chi_{\theta}$. Next, we recall the Gau\ss\, formula: 
\[\overline{\nabla}_XY=\nabla_XY+\sigma_{\theta}(X,Y)=\nabla_XY+\ep^{\theta}\produ{A_{\theta}X,Y}\e^{\theta},\]
for any $X,Y\in TM$. In addition, let $H=\mathrm{tr}(A)/2$, and $H_{\theta}=\mathrm{tr}(A_{\theta})/2$ be the mean curvatures of the immersions $\chi$ and $\chi_{\theta}$, respectively.

\begin{lemma}\label{remark_cmc}
For any $X\in TM$, $JAX+AJX=2HJX$ holds.
\end{lemma}
\begin{proof} Since $\dim M=2$, for any unit vector (field) $X$ on the surface, this yields $2H=\produ{AX,X}+\produ{AJX,JX}$. We readily obtain the lemma.
\end{proof}

\subsection{3-dimensional Homogeneous Manifolds with 4-dim Isometry Group}
We now consider surfaces immersed in a 3-dim homogeneous Riemannian manifold $\mathbb{E}$ whose isometry group has dimension 4. The classification of such manifolds is well-known and depends on two parameters, namely the curvature $\kappa$ of the base surface of the fibration and the bundle curvature $\tau$, where $\kappa$ and $\tau$ are real numbers and $\kappa\neq 4\tau^2$. B.~Daniel gave in \cite{D2} a fundamental theorem for surfaces in such spaces, which we recall.
\begin{theorem} \cite{D2}\label{fundTHhomo}
Let $M$ be a simply connected oriented Riemannian surface with connection $\nabla$, $J$ the rotation angle $\frac{\pi}{2}$ on $TM$, $A$ a self-adjoint $(1,1)$-tensor, $T$ a vector field  on $M$, and $f$ a smooth real valued function such that $\|T\|^2+f^2=1$. Let $\kappa$ and $\tau$ be real numbers such that $\kappa\neq 4\tau^2$. Then there exists an isometric immersion $\chi:M\rightarrow\mathbb{E}$ if, and only if, the data $(A,T,f)$ satisfy the following structure equations, where $K$ is the Gauss curvature of $M$.
\begin{eqnarray}
K&=&\det A+\tau^2+(\kappa-4\tau^2)(1-\|T\|^2) \label{Gausshomo}\\
(\nabla_XA)Y-(\nabla_Y A)X&=&(\kappa-4\tau^2)f(\langle Y,T\rangle X-\langle X,T\rangle Y) \label{Codazzihomo}\\
\nabla_XT&=&f(AX-\tau JX)\label{eqforThomo}\\
X(f)&=&-\langle AX,T \rangle +\langle\tau JX,T\rangle \label{eqforX(f)}
\end{eqnarray}
The operator $A$ turns out to be the shape operator of the immersion and $T$ is the part of $\partial_t$ tangent to the surface $M$.
\end{theorem}

\subsection{Warped products}

The authors proved in \cite{LO} a fundamental theorem for hypersurfaces in some warped products, which we recall for the case of surfaces. We choose numbers $\ep\in\{-1,1\}$,  $c\in\{-1, 0, 1\}$ and $k\in\{0,1\}$. Consider $(\mathbb{M}_k^2(c),g_o)$  the 2-dimensional space form of index $k\in\{0,1\}$ and sectional curvature $c$. Given a smooth positive function $a:\mathbb{M}_k^2\rightarrow (0,\infty)$, we define the warped product $(P^{3}=I\times \mathbb{M}_k^2(c),\produ{,}_1=\ep dt^2+a^2g_o)$, with projection $\pi_I:P^3\rightarrow I$. 
We are going to study Riemannian surfaces in $P^3$. Then,  the number $\varepsilon_3$ will represent the causal character of a normal vector field $e_3$ to the surface, and therefore we can choose among the following options: If $k=0$ and $\varepsilon=+1$, then $\varepsilon_3=+1$; if $k=0$ and $\varepsilon=-1$, then $\varepsilon_3=-1$; if $k=1$, then $\varepsilon=+1$ and $\varepsilon_3=-1$. When $k=0$, $\varepsilon=-1$, $P^3$ cannot contain spacelike surfaces. 
\begin{theorem}\cite{LO} \label{fundtheoremwp}Let $M$ be a simply connected oriented Riemannian surface with Levi-Civita connection $\nabla$, Gauss curvature $K$, a self-adjoint $(1,1)$-tensor $A$,  a vector field $T$ on $M$, and  two smooth real valued functions $f$, $\pi$ on $M$ such that $\|T\|^2+\ep_3f^2=\varepsilon$ and $T=\ep\mathrm{grad}(\pi)$.  Then, there exists an isometric immersion $\chi:M\rightarrow P^{3}$ such that $\pi_I\circ \chi=\pi$ if, and only if, data  $(K,A,T,f)$ satisfy the following structure equations:
\begin{align}
&K=\varepsilon_3\det A-
 \varepsilon\left(\frac{(a')^2}{a^2}-\frac{\ep c}{a^2}  \right)
 -\left(\frac{a''}{a} -\frac{(a')^2}{a^2}+\frac{\ep\,c}{a^2}\right)
 \|T\|^2,\label{Gausswp}\\
&(\nabla_XA) Y - (\nabla_YA) X =\varepsilon_3
\left(\frac{a''}{a} -\frac{(a')^2}{a^2}+\frac{\ep\,c}{a^2}\right)f \Big(\langle Y,T\rangle X -\langle X,T\rangle Y\Big),\label{Codazziwp}\\
&\nabla_XT=fAX+\frac{a'}{a}(X-\ep\langle X,T\rangle T),\label{Tcondwp}\\
 & X(f)=-\produ{AX,T}-\ep\frac{a'}{a} f\langle X,T\rangle, \label{X(f)wp}
\end{align}
for any $X,Y\in TM$, where $a\equiv a\circ \pi$, $a\rq{}\equiv a\rq{}\circ\pi$.
\end{theorem}
We refer to \cite{CX} for the case $c=0$.
In such case, the operator $A$ turns out to be the shape operator
of the immersion, $\ep_3$ is the causal character of the normal vector field along $\chi$, and  
$T$ is the part of $\partial_t$ tangent to the manifold. Equations \eqref{Gausswp} and \eqref{Codazziwp} force to use the notation $c=\pm 1$ or ($c=0$, $\varepsilon=+1$). In the non-flat case $c\neq 0$, this number represents the (normalized) Gaussian curvature of the fibers.  Moreover, notice that equation \eqref{Gausswp} can be rewritten as
\[K=\varepsilon_3\det A-
 \varepsilon\frac{a''}{a}
 +\left(\frac{a''}{a} -\frac{a'^2}{a^2}+\frac{\ep\,c}{a^2}\right)
 (\ep-\|T\|^2).\]
As an additional remark, the authors showed in \cite{LO} that if we set $\eta(X)=\produ{X,T}$, for any $X\in TM$, we get $\mathrm{d}\eta=0$, where $\eta = \ep\mathrm{d}\pi$. In other words, one can replace the condition of the choice of $T$ in Theorem \ref{fundtheoremwp} by the choice of the function $\pi$. In this paper, we choose to use $T$.

\begin{lemma} The (maximal) solutions to the equation $a''a-(a')^2+\ep c=0$ are the following:
\begin{enumerate}
\item If $\ep c=-1$, then $a(t)=C_1^{-1}\cosh(C_1 t+C_2)$, for some $C_1,C_2\in\R$, $C_1\neq 0$.
\item If $\ep c=1$, then $a(t)=\pm t+C_2$ or $a(t)=C_1^{-1}\sin(C_1 t+C_2)$ or $a(t)=C_1^{-1}\sinh(C_1 t+C_2)$, for some $C_1,C_2\in\R$, $C_1\neq 0$.
\item If $c=0$, then $a(t)=C_1\exp(C_2t)$, for some $C_1,C_2\in\R$. 
\end{enumerate}
\end{lemma}
Among the warped products we are considering, it is well-known that those associated with these solutions are isometric to (open subsets) of space forms. Here, the symbol $\cong$ means \textit{isometric to an open subset}. Then,  $-\mathbb{R}\times_{\cosh(t)} \mathbb{S}^2\cong dS^3_1$, the  De Sitter $3$-space;  $\mathbb{R^+}\times_{\sinh(t)} \mathbb{H}^2\cong \mathbb{H}^3\cong\mathbb{R}\times_{\cosh(t)} \mathbb{H}^2$, the Hyperbolic $3$-space; $(0,\pi)\times_{\sin(t)}\mathbb{S}^2\cong \mathbb{S}^3\backslash\{North,South\}$, the round $3$-Sphere without two antipodal points;   $\mathbb{R}^+\times_t\mathbb{S}^2\cong\mathbb{R}^3\backslash\{0\}$ and
$-\mathbb{R}^+\times_t\mathbb{H}^2\cong\mathbb{L}^3$, the Minkowski $3$-space; $\mathbb{R}\times_{e^t}\mathbb{R}^2\cong\mathbb{H}^3$, the  hyperbolic $3$-space; $-\mathbb{R}\times_{e^t}\mathbb{R}^2\cong d{S}_1^3$, the De Sitter $3$-space. Since surfaces in space forms are well understood, we will exclude them in the following discussion. 

\section{Generalized  associated Families}

Let $(M,\produ{,},J)$ be a Riemann surface. We consider $(\mathbb{E},\produ{,})$ a 3-dimensional manifold. Assume that there exists 
an isometric immersion $\chi:(M,\produ{,})\rightarrow (\mathbb{E},\produ{,})$. Let $A$ be the shape operator of the immersion. By Lemma \ref{remark_cmc}, it is very easy to check that
\[ A^c = H\one, \quad A^a=A-H\one, 
\]
where $\one$ is the identity map on $TM$, and $H=\mathrm{tr}(A)/2$ is the mean curvature function. 
\begin{lemma} \label{generalizada} 
Definition \ref{generalized} is equivalent to  
\[ A_{\theta} =  F_1(\theta)\mathrm{e}^{-2J\theta}(A-H\one)+F_2(\theta)H\one,\]
where $\theta\in\R$. 
\end{lemma}

We also need a family of vector fields, $T_{\theta}\in\mathfrak{X}(M)$, $\theta\in\R$. 
\begin{definition} We will say that the family of immersions $\chi_{\theta}:M\rightarrow P$ has 
{\em rotating structure vector field} if there exists a smooth map $\mathbb{R}\times M\rightarrow TM$, $(\theta,p)\mapsto T_{\theta}(p)$ satisfying:
\begin{enumerate}
\item For each $\theta\in\mathbb{R}$, the restriction $T_{\theta}\in\mathfrak{X}(M)$;
\item There exist two smooth functions $\lambda, \mu:\R\rightarrow\R$ such that $\lambda(0)=1$, $\mu(0)=0$, and $T_{\theta}=\lambda(\theta) T+\mu(\theta) JT$. 
\end{enumerate}
\end{definition}
In the following we use the notation $\lambda T+\mu JT=(\lambda\one +\mu J)T$, where $\one$ is the identity map on $TM$. We also need to construct the corresponding family of functions $f_{\theta}:M\rightarrow\R$, $\theta\in\R$, from a map $\mathbb{R}\times M\rightarrow \R$, $(\theta,p)\mapsto f_{\theta}(p)$ such that for each $\theta\in\R$, the restriction $f_{\theta}$ satisfies  the conditions  
\begin{equation}\label{ftheta}
1= \|T_{\theta}\|^2+f_{\theta}^2, \quad 
(\lambda^2+\mu^2)(1-f^2)=1-f_{\theta}^2, \quad 
\theta\in\R, \quad f_0=f, 
\end{equation}
in the homogeneous case and 
\begin{equation}\label{fthetah}
\varepsilon= \|T_{\theta}\|^2+\varepsilon_3f_{\theta}^2, \quad 
(\lambda^2+\mu^2)(\varepsilon-\varepsilon_3f^2)=\varepsilon-\varepsilon_3 f_{\theta}^2, \quad 
\theta\in\R, \quad f_0=f, 
\end{equation}
in the warped product case. Note that in either case, the map $\R\times M\rightarrow \R$, $(\theta,p)\mapsto f_{\theta}(p)$ is always continuous, and smooth whenever it is different from zero. However, we can assume without loss of generality that each $f_{\theta}$ is always smooth.
\begin{lemma} \label{detAtheta}
$\det A_{\theta}=F^2_1\det A+(F^2_2-F^2_1)H^2.$
\end{lemma}
\begin{proof}
Since the determinant of matrices is invariant under rotations, we have 
\begin{align*}\det A_{\theta}&=\det\Big(F_1e^{-2J\theta}(A-H\one)+F_2H\one\Big) \\
&=\det \Big(e^{-J\theta}\Big(F_1(A-H\one)+F_2H\one\Big)e^{J\theta}\Big)\\
&=\det\big(F_1(A-H \one)+F_2 H \one\big).\end{align*}
But using the fact that  $H=\frac{1}{2}\mathrm{tr}(A)$ we get easily
\begin{align*}&\det (F_1(A-H \one)+F_2 H \one)=F_1^2\det A+2F_1(F_2-F_1)H^2+(F_2-F_1)^2H^2\\
&=F_1^2\det A+(F_2-F_1)(2F_1+F_2-F_1)H^2=F^2_1\det A+(F^2_2-F^2_1)H^2.
\end{align*}
\end{proof}
\begin{lemma}\label{htheta}
$H_{\theta}=F_2H$
\end{lemma}
\begin{proof}
$2H_{\theta}=\mathrm{tr}(A_{\theta})=\mathrm{tr}\Big( F_1\mathrm{e}^{-2J\theta}(A-H\one )+F_2H\one
\Big)=F_1\mathrm{tr}(A-H\one)+2F_2H=2F_2H$.
\end{proof}

Let $(E_1,E_2)$ be a parallel local orthonormal frame of $M$. Then, we recall that the divergence of an operator ${\mathcal{T}}$ is given by $\delta\mathcal{T}=\mathrm{tr}(\nabla\mathcal{T})$, or in other words, 
\[\produ{\mathrm{tr}(\nabla \mathcal{T}),X} =\produ{ \delta\mathcal{T},X}=\sum_i \produ{(\nabla_{E_i}\mathcal{T})E_i,X}.
\]
Bearing this in mind, we see 
\begin{eqnarray*} && (\langle\nabla_ {E_1}\mathcal{T}) E_2,E_1\rangle - \langle(\nabla_{E_2}\mathcal{T})E_1,E_1\rangle = \\
&&=\produ{\delta{\mathcal{T}},E_2}-\langle(\nabla_ {E_2}\mathcal{T}) E_2,E_2\rangle- \langle
(\nabla_{E_2}\mathcal{T})E_1,E_1\rangle\\
&&=\produ{\delta {\mathcal{T}},E_2}-E_2(\textrm{tr}(\mathcal{T})).
\end{eqnarray*} 
Similarly we get $(\langle\nabla_ {E_1}\mathcal{T}) E_2,E_2\rangle - \langle (\nabla_{E_2}\mathcal{T})
E_1,E_2\rangle=(\delta\mathcal{T})(E_1)-E_1(\textrm{tr}(\mathcal{T}))$. Therefore, we obtain
$\produ{(\nabla_{E_1}T)E_2-(\nabla_{E_2}T)E_1,X}
=\produ{\delta \mathcal{T},X}-X(\mathrm{tr}(\mathcal{T}))$, 
for any $X\in TM$. Without the vector field, we have 
\begin{equation} \label{deltanabla}
d^{\nabla}\mathcal{T} = \delta \mathcal{T} - \nabla\mathrm{tr}(\mathcal{T}).
\end{equation}
We can apply this formula to $A$ and $A_{\theta}$.\\

%We need to discuss three extreme cases: $T=\dt$, $T=0$ and $0\neq T\neq \dt$ everywhere. 

\section{Homogeneous Spaces}

We recall that a homogeneous space $\mathbb{E}$ satisfying $\tau=0$ reduces to a product $\mathbb{S}^2(r)\times\R$ or $\mathbb{H}^2(-r)\times\R$, for some $r>0$.

\begin{theorem}\label{homo}
Let $\chi:M\rightarrow\mathbb{E}$ be an immersion such that $M$ is connected. Then, $\chi$ admits a generalized associated family with rotating structure vector field if, and only if, $\tau=0$ and $\chi$ is one of the following:
\begin{enumerate}
\item \label{tg} a totally geodesic surface; 
\item \label{mini} a (not totally geodesic) minimal surface;
\item \label{tu} a (not totally geodesic) totally umbilical surface.
\end{enumerate}
\end{theorem}

We split the proof in three extreme cases,  namely open sets on which $T=0$, $0\neq T\neq \partial_t$, $T=\partial_t$. After the following subsections, we obtain $\tau=0$ and the surface will locally be either totally geodesic, minimal or totally umbilical. In the next few lines, we will show that all of these surfaces are analytical. Thus, they are mutually excluding. 

Firstly, any minimal surface in $\mathbb{S}^2(r)\times\R$ or $\mathbb{H}^2(-r)\times\R$ can be locally seen as a graph over an open subset of $\mathbb{S}^2$ or $\mathbb{H}^2$ of a function which is a solution to a well-known elliptic PDE with analytical coefficients. Thus, minimal surfaces in such spaces are analytical. 

Secondly, we recall that totally geodesic surfaces 
are open subsets of either (i) a slice $\mathbb{S}^2\times\{t_o\}$ or $\mathbb{H}^2\times\{t_o\}$, or (ii) vertical cylinders over a geodesic. 

From now until next Subsection \ref{casoTdt}, we will use \cite{ST}, so we suggest the reader to  check it for more details, if necessary. Thirdly, according to \cite{ST}, among the homogeneous spaces $\mathbb{E}$, only the products $\mathbb{S}^2\times\R$ and $\mathbb{H}^2\times\R$ admit totally umbilical surfaces, being open parts of either (i) slices $\mathbb{S}^2\times\{t_o\}$, $\mathbb{H}^2\times\{t_o\}$, or (ii) vertical cylinders over a geodesic in $\mathbb{S}^2$ or $\mathbb{H}^2$, or (iii) umbilical  complete surfaces, which  are invariant by 1-parameter subgroups of ambient isometries. Thus, they are constructed by rotating a profile curve, which makes the whole surface depend just on a certain function $\theta$. 

Next, in $\mathbb{S}^2\times\R$, we assume that we can smoothly glue a non totally geodesic, (rotationally invariant) totally umbilical surface and a minimal surface along a curve. Along such curve, the principal curvatures vanish. In \cite{ST}, page 678, we obtain that  function $\theta$ satisfies $\theta'(s_o)=0=\sin(\theta(s_o))=0$ for some $s_o$ real number. But then, $\cos(\theta(s_o))=\pm 1$. Following the computations on the same page, we see that the totally umbilical side becomes an open subset of a totally geodesic slice. This is a contradiction. Similar computations hold for surfaces in $\mathbb{H}^2\times\R$. 

Since totally geodesic slices can be seen as minimal surfaces in $\mathbb{S}^2\times\R$ and $\mathbb{H}^2\times\R$, and both are analytical, they are mutually excluding. 

On the other hand, the authors of \cite{ST} show that, unlike in real space forms, our totally umbilical surfaces are not CMC. In all cases, the mean curvature function of these surfaces are solutions to several non-trivial ODE, and there exists three families in both $\mathbb{S}^2\times\R$ and $\mathbb{H}^2\times\R$, parametrized on certain intervals. Now take two such surfaces, say $\chi_1$ and $\chi_2$. The respective mean curvatures $H_1$ and $H_2$ are not constant, and $H_1$ cannot be obtained by multiplying $H_2$ by a constant, unless $\chi_1$ and $\chi_2$ are linked by an isometry, and therefore, this constant has to be $1$. See \cite{ST} for details.  Then, taking a generalized associated family with a totally umbilical surface $\chi$, by Lemma \ref{generalizada}, all of them are also totally umbilical, i.~e., $A_{\theta}=F_2(\theta)H\mathbf{1}$.  By the previous argument, $F_2(\theta)=1$ for any $\theta\in\R$, and then $\chi_{\theta}=\Phi_{\theta}\circ\chi$ for some isometry $\Phi_{\theta}$ of either $\mathbb{S}^2\times\R$ or $\mathbb{H}^2\times\R$. Moreover, the functions $\lambda$ and $\mu$ reduce to $\lambda(\theta)=1$ and $\mu(\theta)=0$ for any $\theta\in\R$, since we are not rotating the vector field $T$.

\subsection{The case $T=\dt$} \label{casoTdt}

This condition is equivalent to $f=0$ everywhere.  All these surfaces are known as \textit{vertical cylinders}. The reason is that there exists a curve $\alpha$ on $\mathbb{M}^2(\kappa)$ such that $M=\pi^{-1}(\alpha)$, where $\pi:\mathbb{E}\rightarrow\mathbb{M}^2(\kappa)$ is the natural projection on the fiber. 
By equation \eqref{deltanabla} we get $\delta A=\nabla H$.

By equation \eqref{eqforX(f)}, we get that $A$ has the form $\left( \begin{array}{cc}
0 &-\tau\\
-\tau & 2H
\end{array} \right)$. Notice that $\det A=-\tau^2$. Since $f=0$, then $\nabla_XT=0$ by equation \eqref{eqforThomo}. Consequently
\[0=(\lambda\one+\mu J)\nabla_XT=[1-(\lambda^2+\mu^2)](A_{\theta}X-\tau JX),\]
and either $A_{\theta}=\tau J$, or $1=\lambda^2+\mu^2$. By \eqref{ftheta}, $1=\lambda^2+\mu^2$ is equivalent to $f_{\theta}=0$. If $A_{\theta}=\tau J$, since $A_{\theta}$ is symmetric and $J$ skew-symmetric, then $A_{\theta}=0$ and $\tau=0$. The three equations reduce to
\[ K_{\theta}=\kappa f_{\theta}^2, \quad 
0=f_{\theta}(\produ{T_{\theta},Y}X-\produ{T_{\theta},X}Y), \quad
X(f_{\theta})=0,
\] 
for any $X,Y\in TM$. But then, for each $\theta$, $f_{\theta}$ is a constant function. If for some $\theta$, $f_{\theta}\neq 0$, then $0=\produ{T_{\theta},Y}X-\produ{T_{\theta},X}Y$, which implies 
 $T_{\theta}=0$, and since $\dt=T_{\theta}+f_{\theta}N_{\theta}$, then $f_{\theta}=\pm 1$. Since the map $(\theta,p)\mapsto f_{\theta}(p)$ is continuous and $f_{0}=f=0$, we get a contradiction. Therefore, $f_{\theta}=0$ for any $\theta$. This means that each immersion can be recovered as the pre-image of a curve on the base, as pointed out at the beginning of this section. Moreover, $T_{\theta}=\dt$ for any $\theta$. But now, $A_{\theta} \equiv \left( \begin{array}{cc}
0 & -\tau \\-\tau & 2H_{\theta}
\end{array} \right)$, for any $\theta$. This means \begin{eqnarray*}
-\tau JT_{\theta} = -\tau JT &=&A_{\theta}T_{\theta}=A_{\theta}T
=(F_1\mathrm{e}^{-2J\theta}(A-H\one)+F_2H\one)T 
\\
&=&
H\big((F_2-F_1\cos(2\theta))T+F_1\sin(2\theta)JT\big),
\end{eqnarray*}
and from this, 
\[  0 = H(F_2(\theta)-F_1(\theta)\cos(2\theta)), \quad -\tau =HF_1(\theta)\sin(2\theta),
\]
for any $\theta$ and everywhere on $M$. If $H_p\neq 0$ for some $p\in M$, then it holds  $F_1(\theta)\sin(2\theta)=-\tau/H$. By taking two different values of $\theta$, we see $\tau=0$, and then $F_1=0$, which is a contradiction. Thus, we arrive to $H=0$, and then $\tau=0$. Similarly, by using $JT_{\theta}$, we obtain $H_{\theta}=0$ for any $\theta$.

\subsection{The case $0\neq T\neq \dt$ everywhere}

Note that we have $f\neq 0$ everywhere. Since the map $(\theta,p)\in \R\times M\mapsto f_{\theta}(p)$ is continuous and $f_0\neq 0$, there exist an interval $\tilde{I}\subset \R$ and an open subset $\mathcal{U}\subset M$ such that $f_{\theta}(p)\neq 0$ for any $(\theta,p)\in\tilde{I}\times\mathcal{U}$. In addition, since $F_1(0)=1$, we can also assume that $F_1(\theta)\neq 0$ for any $\theta\in\tilde{I}$. Then, we work on this subset $\tilde{I}\times\mathcal{U}$.

\begin{lemma}\label{cinco} If $f\neq 0$, then the structure equations are equivalent to
\begin{eqnarray}
\det A-\det A_{\theta}&=&(\kappa-4\tau^2)(1-(\lambda^2+\mu^2))(1-f^2) \label{Gausshomonew}\\
f (\delta A_{\theta}-2\nabla H_{\theta})&=&f_{\theta} (\lambda\one+\mu J)(\delta A-2\nabla H) \label{Codazzihomonew}\\
\hspace{-1cm}(\lambda\one+\mu J) \nabla_XT&=&f_{\theta}(F_1e^{-2J\theta}(A-H\one)X+F_2HX-\tau JX)\label{eqforThomonew}
\end{eqnarray}
\end{lemma}
\begin{proof} Formulae \eqref{Gausshomonew} and \eqref{eqforThomonew} are a direct consequence of \eqref{ftheta} and Theorem \ref{fundTHhomo}. Next, we write $Y=k X+m JX$, for some smooth functions $k$ and $m$ defined on open subsets. An easy computation shows that the right hand side is given by
\begin{eqnarray*}
&&\produ{Y,T_{\theta}}X -\produ{X,T_{\theta}}Y=
k\lambda[\langle JX,T\rangle X-\langle X,T\rangle JX]\\
&&\quad +m\,\mu[\langle X, T\rangle X+\langle JX,
T\rangle JX]\\
&&\quad =m(\lambda \one+\mu J)[\langle JX,T\rangle X-\langle X,T\rangle JX]\\
&&\quad =(\lambda \one+\mu J)\Big(\produ{Y,T}X -\produ{X,T}Y\Big).
\end{eqnarray*}
In this way, by \eqref{deltanabla}, we have
\begin{align*}
& f (\delta A_{\theta}-2\nabla H_{\theta})=f (d^{\nabla}A_{\theta}) =
f\,f_{\theta} (\kappa-4\tau^2)( \produ{E_2,T_{\theta}}E_1-\produ{E_1,T_{\theta}}E_2 )\\
& = f\,f_{\theta} (\kappa-4\tau^2)(\lambda\one+\mu J)( \produ{E_2,T}E_1-\produ{E_1,T}E_2 ) \\
& =
f_{\theta}(\lambda\one+\mu J)(d^{\nabla}A)(X)=f_{\theta}(\lambda\one+\mu J)(\delta A-2\nabla H).
\end{align*}
\end{proof}

\begin{lemma}\label{structuremodhomo2}
If $f\neq 0$, the three equations of Lemma \ref{cinco} are equivalent to 
\begin{align}
&(1-F^2_1)(K-\tau^2) - 
(F_2^2-F_1^2)H^2\nonumber \\ & = (\kappa-4\tau^2)\Big(1-(\lambda^2+\mu^2)+(\lambda^2+\mu^2-F_1)f^2\Big),\label{Gauss3}\\
&F_1e^{-2J\theta}\delta A^a-F_2\nabla H=(\lambda\one+\mu J)\frac{f_{\theta}}{f}(\delta A^a-\nabla H),\label{Codazzi3}\\
&\Big(f(\lambda\one+\mu J)-f_{\theta}F_1e^{-2J\theta}\Big)AX \nonumber \\ & =f_{\theta}(F_2-F_1e^{-2J\theta})HX+(f(\lambda\one+\mu J)-f_{\theta})\tau JX \label{T3}.
\end{align}
\end{lemma}
\begin{proof}
By inserting Lemma \ref{detAtheta} in \eqref{Gausshomonew}, we get 
\[(1-F^2_1)\det A-(F^2_2-F^2_1)H^2=(\kappa-4\tau^2)(1-(\lambda^2+\mu^2))(1-f^2).\]
Using again the original Codazzi equation 
$K=\det A+\tau^2+(\kappa-4\tau^2)(f^2) $, we have then
\begin{eqnarray*}
&& (1-F^2_1)(K-\tau^2) -(F_2^2-F^2_1)H^2 \\
&& = (\kappa-4\tau^2)\Big((1-(\lambda^2+\mu^2))(1-f^2)+(1-F_1)f^2\Big)\\
&&=(\kappa-4\tau^2)\Big(1-(\lambda^2+\mu^2)+(\lambda^2+\mu^2-F_1)f^2\Big),
\end{eqnarray*}From \eqref{Codazzihomonew} we immediately get the second equation. Finally, by using equation \eqref{eqforThomonew}, 
\begin{eqnarray*}
&&\nabla_XT_{\theta}=f_{\theta}(F_1e^{-2J\theta}(A-H\one)X+F_2HX-\tau JX)\\
&&=f_{\theta}F_1e^{-2J\theta}AX+f_{\theta}H(F_2\one-F_1e^{-2J\theta})X-f_{\theta}\tau JX\\
&&=f_{\theta}F_1e^{-2J\theta}(\frac{1}{f}\nabla_XT+\tau JX)+f_{\theta}H(F_2\one-F_1e^{-2J\theta})X-f_{\theta}\tau JX\\
&&=\frac{f_{\theta}}{f}F_1e^{-2J\theta}\nabla_XT+f_{\theta}H(F_2\one-F_1e^{-2J\theta})X+f_{\theta}\tau(F_1e^{-2J\theta}-\one) JX
\end{eqnarray*}
and we have
\begin{gather*}
\Big((\lambda\one+\mu J)-\frac{f_{\theta}}{f}F_1e^{-2J\theta}\Big)\nabla_XT \\
=f_{\theta}H(F_2\one-F_1e^{-2J\theta})X+f_{\theta}\tau(F_1e^{-2J\theta}-\one) JX
\end{gather*}
or equivalently
\begin{gather}
\Big(f(\lambda\one+\mu J)-f_{\theta}F_1e^{-2J\theta}\Big)AX\nonumber \\ =f_{\theta}H(F_2\one-F_1e^{-2J\theta})X+\tau (f(\lambda\one+\mu J)-f_{\theta}\one)JX.  \label{eqforThomomodified2}
\end{gather}
\end{proof}
Now plugging in $T$ and $JT$ for $X$ in \eqref{T3}, we get
\hspace{-1cm}\begin{align*}
&\begin{cases}
\Big(f(\lambda\one+\mu J)-f_{\theta}F_1e^{-2J\theta}\Big)AT=f_{\theta}H(F_2\one-F_1e^{-2J\theta})T \\
\hspace{6cm}+\tau(f(\lambda\one+\mu J)-f_{\theta}) JT, \\
\Big(f(\lambda\one+\mu J)-f_{\theta}F_1e^{-2J\theta}\Big)AJT=f_{\theta}H(F_2\one-F_1e^{-2J\theta})JT\\
\hspace{6cm}		-\tau(f(\lambda\one+\mu J)-f_{\theta}) T.
\end{cases}\\
&
\begin{cases}
\Big(f(\lambda\one+\mu J)-f_{\theta}F_1e^{-2J\theta}\Big)AT=f_{\theta}H(F_2\one-F_1e^{-2J\theta})T \\ \hspace{6cm}+\tau(f(\lambda\one+\mu J)-f_{\theta}) JT,\\
\Big(f(\lambda\one+\mu J)-f_{\theta}F_1e^{-2J\theta}\Big)JAJT=-f_{\theta}H(F_2\one-F_1e^{-2J\theta})T\\
\hspace{6cm} -\tau(f(\lambda\one+\mu J)-f_{\theta}) JT.
\end{cases}
\end{align*}
Hence, by adding the two equations, we obtain  
\begin{equation}\label{operators}
\Big(f(\lambda\one+\mu J)-f_{\theta}F_1e^{-2J\theta}\Big)(AT+JAJT)=0.
\end{equation}
Note that this equation holds for any $\theta$.

Let us put $V=AT+JAJT$ and define the operator $B=f(\lambda\one+\mu J)-f_{\theta}F_1e^{-2J\theta}$.

\begin{lemma} If $V=0$ on an open subset $\mathcal{V}$ of $M$, then $\mathcal{V}$ is totally umbilical. 
\end{lemma}
\begin{proof} According to Lemma \ref{remark_cmc}, we know $0=AT+JAJT$ and $JAT+AJT=2HJT$. From here, a simple computation shows $AT=HT$ and $AJT=HJT$. 
\end{proof}
In \cite{Veken}, the classification of such surfaces is obtained. As a result, among 3-dim homogeneous manifolds with 4-dim isometry group, only the products $\mathbb{S}^2\times\mathbb{R}$ and $\mathbb{H}^2\times\mathbb{R}$ admit totally umbilical surfaces. In that paper, the author obtained a full classification, as well as local coordinates of all such surfaces. As a fast description, either they are  totally geodesic or invariant by 1-dim isometry subgroups which also leave invariant the slices of $\mathbb{S}^2\times\mathbb{R}$ and $\mathbb{H}^2\times\mathbb{R}$.

Next, we assume that $V\neq 0$ on $\mathcal{U}$. In fact, $\mathrm{Span}\{V\}\subset \ker B$. Then, we know 
$f(\lambda\one+\mu  J)V=f_{\theta}F_1e^{-2J\theta}V,$ 
which is equivalent to $f\lambda V+f\mu JV =f_{\theta} F_1\cos(2\theta) V -f_{\theta}F_1\sin(2\theta) JV$. 
Since $V$ and $JV$ are linearly independent, we have $f\lambda-f_{\theta}F_1\cos(2\theta) = 0$ and also $f\mu +f_{\theta}F_1\sin(2\theta)=0.$ Since we are assuming $f\neq 0$, we get the following expressions for $\lambda$ and $\mu$: 
\begin{equation}\label{lambdamu}\lambda = \frac{f_{\theta}}{f}F_1\cos(2\theta), \quad \mu=\frac{-f_{\theta}}{f}F_1\sin(2\theta).
\end{equation}
However, by inserting these formulae in $B$, we obtain $B=f(\lambda\one+\mu J)-f_{\theta}F_1e^{-2J\theta} =f\lambda\one+f\mu J-f_{\theta}F_1\cos(2\theta)\one+f_{\theta}F_1\sin(2\theta)J=0.$ 
Equation \eqref{eqforThomomodified2} becomes 
$0=BAX =f_{\theta}(F_2-F_1e^{-2J\theta})HX+(f_{\theta}F_1e^{-2J\theta}-f_{\theta}\one)\tau JX$, 
and since $f_{\theta}$ is not $0$, we get $(F_2-F_1e^{-2J\theta})HX+(F_1e^{-2J\theta}-\one)\tau JX=0$, 
and therefore
\begin{equation}\label{relationHandtau}
\begin{array}{c}
(F_2-F_1\cos(2\theta))H+F_1\sin(2\theta)\tau=0,\\
F_1\sin(2\theta)H+(F_1\cos(2\theta)-1)\tau=0.
\end{array}
\end{equation}
These two equations hold for any $\theta\in\tilde{I}$. By taking two different values for $\theta$, but close enough, we obtain that $H=\tau = 0$. 
 Coming back to \eqref{lambdamu}, we see that for any $X\in TM$,
\[  X\Big(\frac{f_{\theta}}{f}\Big) F_1\cos(2\theta) = 0 =  X\Big(\frac{f_{\theta}}{f}\Big) F_1\sin(2\theta).\]
Since $F_1\neq 0$, there exists a function $b(\theta)$ defined for $\theta\in\tilde{I}$ such that   $f_{\theta}=bf$. This means that $\lambda= bF_1\cos(2\theta)$ and $\mu=-bF_1\sin(2\theta)$, which leads to $T_{\theta}=bF_1e^{-2J\theta}T$. But now, $b^2f^2=f_{\theta}^2=1-\|T_{\theta}\|^2=1- b^2F_1^2\|T\|^2 = 1-b^2F_1^2(1-f^2)=1-b^2F_1^2+b^2F_1^2f^2$. This means that $f$ is constant on $\mathcal{U}$, and so is $f_{\theta}$ for each $\theta$, or $b=F_1=1$.  Again, in the first case, we see by \eqref{eqforX(f)} that $A_{\theta}T_{\theta}=0$. Since $M$ is a surface, we have $A_{\theta}JT_{\theta} = 2H_{\theta}JT_{\theta}=2HF_2JT_{\theta}=0$. In particular, $A=0$ and $M$ is totally geodesic. In the second case, the  associated family is the minimal family discussed by Daniel in \cite{D2}. 
%It is not a problem to include the totally geodesic surfaces among the totally umbilical ones. 

\subsection{The case $T=0$ everywhere}

Now, $f^2=1$, so by \eqref{eqforThomo}, we see $0=AX-\tau JX$ for any $X\in TM$. Since $A$ is symmetric and $J$ is skew-symmetric, then $A=0$ and $\tau=0$. This means that we are in the product case and   $M$ is an open subset of either $\mathbb{S}^2$ or $\mathbb{H}^2$ embedded in the ambient space as a totally geodesic slice. 

\section{Warped Product Spaces}

\begin{theorem}Let  $\chi:M^2\rightarrow (P^{3}=I\times \mathbb{M}_k^2(c),\produ{,}_1=\ep dt^2+a^2g_o)$, be an isometric immersion, where $\ep=\pm 1$. Assume that $P^3$ does not contain any open subset with constant sectional curvature. Then, $\chi$ admits a generalized  associated family with rotating structure vector field if, and only if, $M$ admits an open dense $\Omega\subset M$ such that $\chi$ restricted to each connected component of $\Omega$ is one of the following cases: 
\begin{enumerate}
\item Function $a$ is a constant function, and the surface is a vertical cylinder over a geodesic in $\mathbb{M}^2_k(c)$;
\item function $a$ is a constant function,  and the surface is minimal;
\item the surface is totally umbilical.
\end{enumerate}
\end{theorem}

Note that slices $\{t_o\}\times \mathbb{M}^2_k(c)$ are totally umbilical in $P^3$. Also, totally geodesic submanifolds can be regarded as either minimal or totally umbilical. 

We split the proof in three extreme cases, namely $T=0$, $0\neq T\neq \partial_t$, $T=\partial_t$.  After that, 
there will be a dense open $\Omega\subset M$ such that each connected component of $\Omega$ will be one of the cases in the list of the Theorem. When $a$ is a constant function, the manifold $P^3$ becomes analytic, and so its minimal and totally geodesic surfaces, and vertical cylinders over geodesics are also analytic. However, the situation is not so satisfactory when the surface is totally umbilical, as pointed out in \cite{SV}.

\subsection{The case $T=\dt$}
 
By equation \eqref{X(f)wp},  we get that $\langle AT,X\rangle=0$ for any $X$ tangent to $M$, thus $A$ has to have the form $\left( \begin{array}{cc}
0 & 0 \\
0 & 2H
\end{array} \right)$. Notice then that $A^a=\left( \begin{array}{cc}
H & 0 \\
0 & -H
\end{array} \right)$ and $\det A=0$. 

Firstly, we assume that $a\rq{}\neq 0$. 
From equation \eqref{Tcondwp} we get 
\[
(\lambda\one+\mu J)\nabla_XT=(\lambda\one+\mu J)\frac{a'}{a}(X-\varepsilon\produ{X,T}T)=\frac{a'}{a}(X-\varepsilon\produ{X,T_{\theta}}T_{\theta})).\]
By inserting $X=T$, we have $(\lambda\one+\mu J)\frac{a'}{a}(T-\varepsilon\produ{T,T}T)=0$, and consequently since $a\rq{}\neq 0$, $0=T-\ep\produ{T,T_{\theta}}T_{\theta}=(1-\lambda^2)T-\lambda\mu JT$, which readily implies $\lambda=1$, $\mu=0$ and $T=T_{\theta}$ for any $\theta$. From here, $f_{\theta}=0$ for any $\theta$, so that we repeat the computations to get $A_{\theta}T_{\theta}=0$. By using the expression of $A_{\theta}$, we obtain $0=A_{\theta}T_{\theta}=-F_1H(\cos(2\theta)-\sin(2\theta)JT)+F_2HT$, which means $F_1H\cos(2\theta)=0$. Clearly, $H=0$ and by Lemma \ref{htheta}, then $H_{\theta}=0$ for any $\theta$. In other words, $A_{\theta}=0$ for any $\theta$. 

Secondly, we assume there is a connected open subset $U$ of $M$ such that $a'\circ\pi_I(p)=0$ for any $p\in U$. By shrinking $U$ if necessary, we have that $\chi(U)\subset\{t_o\}\times \mathbb{M}_k^2(c)$, that is to say, $U$ is mapped onto a slice. In such case, on $U$, $T=\dt$ is normal to the surface, which is a contradiction. 

Thirdly, we can assume that $a'=0$ (on an open interval). From the structure equations we obtain $\nabla_{X}T_{\theta} = f_{\theta}A_{\theta} X,$ for any $\theta$ and any $X\in TM$. In particular, for $\theta=0$, we obtain $\nabla_XT=0=\nabla_XJT$. This means $f_{\theta}A_{\theta} X = \nabla_{X}T_{\theta} = \lambda \nabla_{X}T+\mu \nabla_{X}T =0$. Next, $f_{\theta}X(f_{\theta}) = -f_{\theta}\produ{A_{\theta}X,T_{\theta}}=0$, which implies $X(f_{\theta}^2)=0$. This shows that $f_{\theta}$ is a constant function. As in the homogeneous case, by the continuity of $(\theta,p)\mapsto f_{\theta}(p)$, we obtain that $f_{\theta}=0$ for any $\theta$. Next, from \eqref{X(f)wp}, we see $A_{\theta}T_{\theta}=0$. We repeat the computations as in the case $a\rq{}\neq 0$ to obtain $A_{\theta}=0$ for any $\theta$. 

\subsection{The case $0\neq T\neq \dt$ everywhere}

With computations similar to the case of homogeneous spaces we obtain easily the analog of Lemma \ref{cinco} in the case of warped products. 
\begin{lemma}\label{structureeqmod1wp} If $f\neq 0$, then the structure equations are equivalent to
\begin{eqnarray}
\det A-\det A_{\theta}&=&\varepsilon_3\left(\frac{a''}{a} -\frac{a'^2}{a^2}+\frac{\ep\,c}{a^2}\right)(1-(\lambda^2+\mu^2))(\varepsilon-\varepsilon_3f^2) \label{Gausshomonew2}\\
f (\delta A_{\theta}-2\nabla H_{\theta})&=&f_{\theta} (\lambda\one+J\mu)(\delta A-2\nabla H) \label{Codazziwpnew}\\
\hspace{-1cm}(\lambda\one+\mu
J)\nabla_XT&=&f_{\theta}(F_1e^{-2J\theta}(A-H\one)X+F_2HX)\nonumber \\
&&+\frac{a'}{a}(X-\ep\langle X,T_{\theta}\rangle T_{\theta}).\label{eqforTwpnew}
\end{eqnarray}
\end{lemma}
Using similar arguments, Lemma \ref{structuremodhomo2} becomes 
\begin{lemma}\label{structuremodwp2}
\begin{align}
&(1-F^2_1)\Big(K+\varepsilon\frac{a''}{a}\Big)+\varepsilon_3(F_2^2-F_1^2)H^2\nonumber\\&\hspace{0.5cm}=\left(\frac{a''}{a} -\frac{a'^2}{a^2}+\frac{\ep\,c}{a^2}\right)\Big(\ep(1-(\lambda^2+\mu^2))+(\lambda^2+\mu^2-F_1^ 2)\varepsilon_3f^2\Big),\label{Gausswp3}\\
&F_1e^{-2J\theta}\delta A^a-F_2\nabla H=(\lambda\one+\mu J)\frac{f_{\theta}}{f}(\delta A^a-\nabla H),\label{Codazziwp3}\\
&\Big(f(\lambda\one+\mu J)-f_{\theta}F_1e^{-2J\theta}\Big)AX=f_{\theta}(F_2-F_1e^{-2J\theta})HX\nonumber\\
&\hspace{0.5cm}+\frac{a'}{a}\Big(X-\ep\langle X,T_{\theta}\rangle T_{\theta})-(\lambda\one+\mu J)(X-\ep\langle X,T\rangle T)\Big), \label{Twp3}
\end{align}
\end{lemma}

\begin{proof}  We used $K=\varepsilon_3\det A- \varepsilon\frac{a''}{a}
 +\left(\frac{a''}{a} -\frac{a'^2}{a^2}+\frac{\ep\,c}{a^2}\right)
 (\ep_3f^2)$ for equation \eqref{Gausswp3}. From \eqref{eqforTwpnew}, we have 
\begin{align*}
(\lambda\one+\mu J)\Big(fAX+\frac{a'}{a}( X-\varepsilon\langle X,T\rangle T)\Big)=&
f_{\theta}(F_1e^{-2J\theta}(A-H\one)X+F_2HX)\\
&+\frac{a'}{a}(X-\ep\langle X,T_{\theta}\rangle T_{\theta}).
\end{align*}
From here, we easily obtain \eqref{Twp3}. 
\end{proof}

Now, pluging in $T$ and $JT$ for $X$ in the last equation we get 
\begin{align*}
\begin{cases}
\Big(f(\lambda\one+\mu J)-f_{\theta}F_1e^{-2J\theta}\Big)AT=f_{\theta}(F_2-F_1e^{-2J\theta})HT\\
+\frac{a'}{a}\Big((1-(\lambda\one+\mu J))T-\ep\lambda\|T\|^2 (\lambda \one+\mu J )T+(\lambda\one+\mu J)(\ep\|T\|^2T)\Big),\\
\Big(f(\lambda\one+\mu J)-f_{\theta}F_1e^{-2J\theta}\Big)AJT=f_{\theta}(F_2-F_1e^{-2J\theta})HJT\\
+\frac{a'}{a}\Big((1-(\lambda\one+\mu J))JT-\ep\mu\|T\|^2 (\lambda \one+\mu J )T\Big),
\end{cases}\end{align*}
and consequently
\begin{align*}
\begin{cases}
\Big(f(\lambda\one+\mu J)-f_{\theta}F_1e^{-2J\theta}\Big)JAT=f_{\theta}(F_2-F_1e^{-2J\theta})HJT\\\hspace{3cm}+\frac{a'}{a}\Big(1-(\lambda\one+\mu J)+\ep(1-\lambda)(\lambda\one +\mu J)\|T\|^2\Big)JT, \\
\Big(f(\lambda\one+\mu J)-f_{\theta}F_1e^{-2J\theta}\Big)AJT=f_{\theta}(F_2-F_1e^{-2J\theta})HJT\\
\hspace{3cm}+\frac{a'}{a}\Big(1-(\lambda\one+\mu J)+\ep J\mu(\lambda\one+\mu J )\|T\|^2\Big)JT.
\end{cases}
\end{align*}
Subtracting these formulas we get
\begin{gather}
\Big(f(\lambda\one+\mu J)-f_{\theta}F_1e^{-2J\theta}\Big)(JAT-AJT) \nonumber \\
=\frac{a'}{a}\Big((\lambda\one+\mu J )\ep(1-\lambda-J\mu)\|T\|^2\Big)JT
\Big). \label{addingTeq}
\end{gather}
And adding them, 
\begin{gather*}
2\big(f(\lambda\one+\mu J)-f_{\theta}F_2\big)HJT \\
=\frac{a'}{a}\Big(2(1-(\lambda\one+\mu J))JT+\ep\big((\lambda\one +\mu J)-(\lambda^2+\mu^2)\big)\|T\|^2JT\Big).
\end{gather*}
Hence we obtain $\mu\Big(\frac{a'}{a}\big(2-\ep\|T\|^2 )+2f H\Big)=0.$ Moreover
\begin{align*}
2(f\lambda -f_{\theta}F_2) H&=\frac{a'}{a}\Big(2(1-\lambda)+\ep(\lambda-(\lambda^2+\mu^2)\|T\|^2\Big)\\
&=\frac{a'}{a}\Big((1-\lambda)\big[2-\ep\|T\|^2\big]+\ep\big[1-(\lambda^2+\mu^2)\big]\|T\|^2\Big)
\end{align*} 
and finally we get the two equations
\begin{align}
&\mu\Big(\frac{a'}{a}\big(2-\ep\|T\|^2 )+2f H\Big)=0\label{eqforH1}\\
2(f\lambda -f_{\theta}F_2) H&=\frac{a'}{a}\Big((1-\lambda)\big[2-\ep\|T\|^2\big]+\ep\big[1-(\lambda^2+\mu^2)\big]\|T\|^2\Big)\label{eqforH2},
\end{align}

\subsubsection{Case $\mu\neq 0$}
If $\mu $ is not $0$, then $ 2f H=-\frac{a'}{a}\big(2-\ep\|T\|^2)$ and $2(f -f_{\theta}F_2) H=\frac{a'}{a}\ep\big(1-\lambda^2-\mu^2\big)\|T\|^2$. Consequently, writing $z_{\theta}:=\lambda\one+\mu J$, we have 

\begin{equation} \label{paluego} 
2f_{\theta}F_2H=-\frac{a'}{a}\big(2-\ep|z_{\theta}|^2\big)\|T\|^2,
\end{equation}
and so 
\[
a\rq{} (2-\ep\|T\|^2)f_{\theta}F_2 = a\rq{}f(2-\ep(\lambda^2+\mu^2)\|T\|^2).
\]

As already seen, the case $a\rq{}=0$ means that either $M$ is contained in a slice or $a$ is constant.

Assuming $a\rq{}\neq 0$, we have $(2-\ep\|T\|^2)f_{\theta}F_2 = f(2-\ep(\lambda^2+\mu^2)\|T\|^2)$. 
Since we have $T\neq 0$, we suppose for a moment that $2-\ep\|T\|^2=0$ at some point of $\mathcal{U}$. Thus,  by \eqref{eqforH1}, we see $H=0$. We get $0=2-\varepsilon(\lambda^2+\mu^2)\|T\|^2 =2(1-\lambda^2-\mu^2)$. This means $\lambda^2+\mu^2=1$.  
Firstly, if $\varepsilon=+1$, then $\|T\|^2=2$, but we recall that $\partial_t=T+fN$, so that $1=\|T\|^2+f^2=2+f^2$, which is a contradition. Secondly, if $\varepsilon=-1$, then $\|T\|^2=-2$. Inserting all the information in \eqref{paluego}, $0=-6$, another contradiction. 
\\

Then, we can assume that $\|T\|^2\neq 2\varepsilon$. 
By the previous section, we can discard the case $f_{\theta}=0$. We arrive to 
\begin{align}\label{eqF_2wp}
F_2=\frac{f}{f_{\theta}}\frac{(2-\ep|z_{\theta}|^2\|T\|^2)}{(2-\ep\|T\|^2)}.
\end{align}
Now plugging $F_2$ in equation \eqref{Codazziwp3} yields 
\begin{align} \label{painful_eq} F_1e^{-2J\theta}\delta A^a-\frac{f}{f_{\theta}}\frac{(2-\ep|z_{\theta}|^2\|T\|^2)}{(2-\ep\|T\|^2)}\nabla H=z_{\theta}\frac{f_{\theta}}{f}(\delta A^a-\nabla H)
\end{align}
Let now $W:=JAT-AJT$. \\

If $W=0$ on an open subset $\mathcal{V}$ of $M$, then by Lemma \ref{remark_cmc}, $\mathcal{V}$ is totally umbilical.  We notice that totally umbilical surfaces which are neither vertical nor horizontal in (warped) products of the form $M^n\times_f I$, with $M^n$ a Riemannian manifold, have been studied and classified in \cite{SV}. In particular the authors prove that such surfaces exist if, and only if, $M^n$ has locally the structure of a warped product, which ensures their existence in our case.

Next, we assume that $\mathcal{U}$ is free of umbilical points. Then, there exist two smooth functions $\alpha,\beta$ defined on $\mathcal{U}$ such that $(\alpha\one +\beta J) W=JT$. Then, by \eqref{addingTeq}  we have
\begin{align*}
(fz_{\theta}-f_{\theta} F_1e^{-2J\theta})W=\frac{a'}{a}\left(z_{\theta}\ep(1-z_{\theta})\|T\|^2\right)(\alpha\one +\beta J) W.
\end{align*}
and consequently $fz_{\theta}-\frac{a'}{a}\left(z_{\theta}\ep(1-z_{\theta})\|T\|^2\right)(\alpha\one +\beta J)=f_{\theta} F_1e^{-2J\theta}$. Hence replacing the coefficient of $\delta A_a$ on the right handside of equation \eqref{painful_eq} we get
\begin{align*}
f_{\theta}F_1e^{-2J\theta}\delta A^a=-\frac{a'}{a}\Big(z_{\theta}\ep(1-z_{\theta})\|T\|^2\Big)(\alpha\one+\beta J)\delta A^a+fz_{\theta}\delta A^a
\end{align*}
and therefore
\begin{align}\label{polynomial_in_z}
f^2&\frac{(2-\ep|z_{\theta}|^2\|T\|^2)}{(2-\ep\|T\|^2)}\nabla H+\ep\ep_3z_{\theta}(1-\ep|z_{\theta}|^2\|T\|^2)(\delta A^a-\nabla H)\\&=-\frac{a'f}{a}\Big(z_{\theta}\ep(1-z_{\theta})\|T\|^2\Big)(\alpha\one+\beta J)\delta A^a+f^2z_{\theta}\delta A^a, \nonumber
\end{align}
which is a cubic polynomial of the form $P(z_{\theta},\bar{z_{\theta}})=c_0+c_1z_{\theta}+c_2|z_{\theta}|^2+c_3z_{\theta}^2+c_4|z_{\theta}|^2z_{\theta}=c_0+z_{\theta}(c_1+c_2\bar{z_{\theta}}+c_3z_{\theta}+c_4|z_{\theta}|^2)=0.$ 
Now,  we want $z_{\theta}$ to be a smooth family of solutions. In order for the solution set $z_{\theta}$ to contain a curve, $P$ needs to share a common factor with its conjugate.  But, in that case either the polynomial is identically 0, or it is easy to check that the polynomial must have the following irreducible factorization $P(z_{\theta},\bar{z_{\theta}})=(z_{\theta}+d_0)(d_1+d_2z_{\theta}+d_3|z_{\theta}|^2)$. The first factor cannot be a multiple of its conjugate, and the second factor is a multiple of its conjugate if, and only if, $d_2=0$, and consequently the curve is a circle of radius $r=\sqrt{-\frac{d_1}{d_3}}$ centered at the origin. Notice that the polynomial could also be quadratic, but then the same term in $z_{\theta}^2$ has to vanish, additionally to the terms in $z_{\theta}$ and $|z_{\theta}|^2z_{\theta}$, so that we can reduce our study to the previous case. Now by equation \eqref{polynomial_in_z} \[d_2=\frac{a'f}{a}\ep\|T\|^2(\alpha\one+\beta J)\delta A^a=0.\]
But this is satisfied if, and only if, one of the following two cases hold.
\begin{enumerate}
\item $a'=0$ and the ambient manifold is in fact a product. Moreover,  by equation \eqref{eqforH1}, we have that $H=0$. In this case we get from equations \eqref{Gausshomonew2} and \eqref{Gausswp3} that $\det A=\det A_{\theta}$, which by Lemma \ref{detAtheta} is equivalent to having $F_1=1$ and finally by equations \eqref{Codazziwp3} and \eqref{Twp3} we can conclude that there exists a family if, and only if, $f_{\theta}=f$ and $\lambda\one+\mu J=e^{-2J\theta}$, which is exactly the usual  associated family used by Eschenburg. Conversely we see easily that if $A_{\theta}=e^{-J\theta}Ae^{J\theta}$ and $T_{\theta}=e^{-2J\theta}T$ and $H=0$, the structure equations are all satisfied, recovering hence Daniel's minimal family when the warped product is Riemannian (see \cite{D2}) and Roth's result when the warped product is Lorentzian (see \cite{R}).

\item $\delta A^a=0$, then by equation \eqref{Codazziwp3} $F_2\nabla H=(\lambda\one+\mu J)\frac{f_{\theta}}{f}\nabla H$, which means that $H$ is constant since $\mu\neq 0$ and $f\neq 0$. Hence it is easy to see that $d^{\nabla}A=0$. By the Codazzi equation in Theorem \ref{fundtheoremwp} this holds if, and only if, $(\frac{a''}{a}-\frac{a'^2}{a^2}+\frac{\ep c}{a^2})=0$ or $\|T\|=0$. Since $T\neq 0$, the remaining situation is when $a''a-(a')^2+\ep c=0$, that is to say, the case of space forms that we excluded at the beginning.
\end{enumerate}
\subsubsection{Case $\mu =0$ }
In this case
\[\lambda\nabla_XT=\lambda\frac{a'}{a}(X-\varepsilon\produ{X,T}T)=\frac{a'}{a}(X-\varepsilon\produ{X,T_{\theta}}T_{\theta}))=\frac{a'}{a}(X-\varepsilon\lambda^2\produ{X,T}T).\]
Hence pluging in $T$ and $JT$ for $X$ we get
\[(\lambda-1)\frac{a'}{a} T=\frac{a'}{a}\varepsilon\lambda(1-\lambda)\|T\|^2T,\quad
 \frac{a'}{a} \lambda JT=\frac{a'}{a} JT\]
and consequently either $a'=0$ and the ambient manifold is a product or $\lambda=1$. 
But in both cases we get from equation \eqref{addingTeq} that either the vector field $W=AJT-AJT=0$, which implies that $M$ is totally umbilical, or $\lambda f=f_{\theta}F_1\cos(2\theta)$ and $0=f_{\theta}F_1\sin(2\theta)$ for all angles $\theta$ in an interval around zero, which is a contradiction.

\subsection{The case $T=0$}

Since $f=\pm 1$ globally, we immediately obtain $\partial_t=\pm N$. In other words, $\chi$ is a slice. And all slices are totally umbilical. 

\section*{Acknowledgments}

The second author has been partially financed by the Spanish Ministry of Economy and Competitiveness and European Regional Development Fund (ERDF), project MTM2016-78807-C2-1-P. In addition, the authors would like to thank the referee for some useful comments.

\end{document}